\documentclass[11pt,a4paper]{article}
\usepackage[latin1]{inputenc}
\usepackage{amsmath}
\usepackage{amsthm}
\usepackage{amsfonts}
\usepackage{amsfonts,amsthm,latexsym,amsmath,amssymb,amscd,epsfig,psfrag,enumerate}
\usepackage{graphics,graphicx, bezier, float, color, hyperref}
\usepackage{amssymb,url}
\usepackage{multienum}
\usepackage[table]{xcolor}
\usepackage{multicol,multirow}
\usepackage{graphicx}
\usepackage{fancyvrb}
\usepackage{parskip}
\usepackage[toc,page]{appendix}
\usepackage[top=2.7cm, bottom=2.7cm, left=1.5cm, right=1.5cm]{geometry}
\usepackage{xcolor}

\usepackage{blkarray}
\newtheorem{theorem}{Theorem}[section]
\newtheorem{lemma}{Lemma}[section]

\usepackage[none]{hyphenat}[section]
\newtheorem{definition}{Definition}[section]

\numberwithin{equation}{section}
\numberwithin{table}{section}
\numberwithin{figure}{section}

\title{On Wagstaff primes in the $k$-Lucas number sequence}
\author{Herbert Batte$^{1,*} $ }
\date{}

\begin{document}
\maketitle
\abstract{A Wagstaff prime is a prime number of the form $(2^{\mathfrak{p}}+1)/3$, where $\mathfrak{p}$ is an odd prime. Let $(L_n^{(k)})_{n\geq 2-k}$ be the $k$-Lucas number sequence defined by the recurrence relation
	$ L_n^{(k)} = L_{n-1}^{(k)} + \cdots + L_{n-k}^{(k)}$, for all $n \ge 2$, with initial terms \( L_0^{(k)} = 2 \) and \( L_1^{(k)} = 1 \) for all \( k \ge 2 \), and \( L_{2-k}^{(k)} = \cdots = L_{-1}^{(k)} = 0 \) for \( k \ge 3 \). In this paper, we show that the only solutions to the Diophantine equation $L_n^{(k)} = (2^{\mathfrak{p}}+1)/3$ are $(n,k,\mathfrak{p})\in\{(5,2,5),(6,4,7)\}\cup \{(2,k,3):k\ge 2\}$. We use linear forms in logarithms and the LLL reduction method to prove our result.} 



\thanks{$ ^{*} $ Corresponding author}

\section{Introduction}\label{intro}
\subsection{Background}\label{sec:1.1}
For an integer $k \ge 2$, the sequence of $k$-generalized Lucas numbers $(L_n^{(k)})_{n \ge 2-k}$ is defined by the recurrence relation
\begin{equation*}
	L_n^{(k)} = L_{n-1}^{(k)} + L_{n-2}^{(k)} + \cdots + L_{n-k}^{(k)}, \quad \text{for all} \ n \ge 2,
\end{equation*}
with the initial conditions $L_0^{(k)} = 2$, $L_1^{(k)} = 1$, and $L_j^{(k)} = 0$ for $2-k \le j \le -1$. When $k=2$, we get the classical Lucas numbers $2, 1, 3, 4, 7, 11, \dots$. 

In Number theory, a Wagstaff prime is a prime of the form 
\begin{equation*}
	W = \frac{2^{\mathfrak{p}} + 1}{3},
\end{equation*}
where $\mathfrak{p}$ is an odd prime. These primes are named after the mathematician Samuel S. Wagstaff Jr. and are related to Mersenne primes and perfect numbers. Some examples of Wagstaff primes include 3, 11, 43, and 683 (see sequence A000979 in \cite{OEIS}).

The search for special types of prime numbers within linear recurrence sequences is a popular topic. Recently, Rezaiguia et al. \cite{REZ} studied the Diophantine equation $F_n^{(k)} = (2^{\mathfrak{p}}+1)/3$, where $F_n^{(k)}$ is the $k$-generalized Fibonacci sequence. They proved that 3 is the only Wagstaff prime that appears in those sequences. 
Motivated by their work, we ask the same question for the $k$-Lucas number sequence. Since $k$-Lucas numbers grow at a similar rate to $k$-Fibonacci numbers but start with different initial values, it is natural to check if they contain more Wagstaff primes. In this paper, we solve the Diophantine equation
\begin{equation}\label{eq:main}
	L_n^{(k)} = \frac{2^{\mathfrak{p}} + 1}{3},
\end{equation}
in non-negative integers $n, k, \mathfrak{p}$ where $\mathfrak{p}$ is an odd prime and $k \ge 2$. Since $L_0^{(k)}=2$ and $L_1^{(k)}=1$ for all $k\ge 2$, and these are not Wagstaff primes, we may assume $n\ge 2$ throughout the paper.

We prove the following result.
\subsection{Main Result}\label{sec:1.2}
\begin{theorem}\label{thm1.1} 
	Let $(L_n^{(k)})_{n\geq 2-k}$ be the sequence of $k$-generalized Lucas numbers. Then, the only integer solutions $(n,k,\mathfrak{p})$ to Eq. \eqref{eq:main} with $n\ge 0$, $k\ge 2$ and prime $\mathfrak{p}\ge 3$ are
	\begin{align*}
		(n,k,\mathfrak{p})\in\{(5,2,5),(6,4,7)\}\cup \{(2,k,3):k\ge 2\}.
	\end{align*}
	Specifically, $L_2^{(k)}=3=(2^{3}+1)/3$, $L_5^{(2)}=11=(2^{5}+1)/3$ and $L_6^{(4)}=43=(2^{7}+1)/3$ are the only solutions to \eqref{eq:main}. 
\end{theorem}

\section{Methods}
\subsection{Preliminaries}
It is known that 
\begin{align}\label{eq2.2}
	L_n^{(k)} = 3 \cdot 2^{n-2},\qquad \text{for all}\qquad 2 \le n \le k.
\end{align}
Additionally, $L_{k+1}^{(k)}=3\cdot 2^{k-1}-2$ and by induction one proves that 
\begin{equation}
\label{eq:32}
L_n^{(k)}<3\cdot 2^{n-2}\qquad {\text{\rm holds for all}}\qquad n\ge k+1.
\end{equation}
Next, we revisit some properties of the $k$--generalized Lucas numbers. They form a linearly recurrent sequence of characteristic polynomial
\[
\Psi_k(x) = x^k - x^{k-1} - \cdots - x - 1,
\]
which is irreducible over $\mathbb{Q}[x]$. The polynomial $\Psi_k(x)$ possesses a unique real root $\alpha(k)>1$ and all the other roots are inside the unit circle, see \cite{MIL}. The root  $\alpha(k):=\alpha$ is in the interval
\begin{align}\label{eq2.3}
2(1 - 2^{-k} ) < \alpha < 2
\end{align}
as noted in \cite{WOL}. As in the classical case when $k=2$, it was shown in \cite{BRL} that 
\begin{align}\label{eq2.4}
	\alpha^{n-1} \le L_n^{(k)}\le2\alpha^n, \quad \text{holds for all} \quad n\ge1, ~k\ge 2.
\end{align}
In particular, combining \eqref{eq2.4} and \eqref{eq:main}, we have
\begin{align*}
	\alpha^{n-1} \le L_n^{(k)}=\frac{2^{\mathfrak{p}} + 1}{3}< 2^{\mathfrak{p}} + 2< 2^{\mathfrak{p}+1},
\end{align*}
for all odd primes $\mathfrak{p}$. Taking logarithms both sides, we get $(n-1)\log \alpha < (\mathfrak{p}+1)\log 2$, or equivalently $n<2\mathfrak{p}+3$. On the other hand, 
\begin{align*}
	2\alpha^{n} \ge L_n^{(k)}=\frac{2^{\mathfrak{p}} + 1}{3} > \frac{2^{\mathfrak{p}}}{3} > 2^{\mathfrak{p}-2},
\end{align*}
holds for all odd primes $\mathfrak{p}$. Taking logarithms both sides, we get $n\log \alpha > (\mathfrak{p}-3)\log 2$, or equivalently $n > \mathfrak{p}-3$, where we have used the fact that $\alpha<2$ from \eqref{eq2.3}. Therefore, we have
\begin{align}\label{m_b}
	\mathfrak{p}-3 <n<2\mathfrak{p}+3.
\end{align}

Next, let $k\ge 2$ and define
\begin{equation*}
	f_k(x):=\dfrac{x-1}{2+(k+1)(x-2)}.
\end{equation*}
We have 
$$
\frac{df_k(x)}{dx}=-\frac{(k-1)}{(2+(k+1)(x-2))^2}<0,\qquad {\text{\rm for~all}}\qquad x>0.
$$
In particular, inequality \eqref{eq2.3} implies that
\begin{align}\label{eq2.5}
	\dfrac{1}{2}=f_k(2)<f_k(\alpha)<f_k(2(1 - 2^{-k} ))\le \dfrac{3}{4},
\end{align}
for all $k\ge 3$. It is easy to check that the above inequality holds for $k=2$ as well. Further, it is easy to verify that $|f_k(\alpha_i)|<1$, for all $2\le i\le k$, where $\alpha_i$ are the remaining 
roots of $\Psi_k(x)$ for $i=2,\ldots,k$.

Moreover, it was shown in \cite{BRL} that
\begin{align}\label{eq2.6}
	L_n^{(k)}=\displaystyle\sum_{i=1}^{k}(2\alpha_i-1)f_k(\alpha_i)\alpha_i^{n-1}~~\text{and}~~\left|L_n^{(k)}-f_k(\alpha)(2\alpha-1)\alpha^{n-1}\right|<\dfrac{3}{2},
\end{align}
for all $k\ge 2$ and $n\ge 2-k$. This means that 
\begin{equation}
\label{eq:Lnwitherror}
L_n^{(k)}=f_k(\alpha)(2\alpha-1)\alpha^{n-1}+e_k(n), \qquad {\text{\rm where}}\qquad |e_k(n)|<1.5.
\end{equation} 
The left expression in \eqref{eq2.6} is known as the Binet formula for $L_n^{(k)}$. Furthermore, the right inequality expression in \eqref{eq2.6} shows that the contribution of the zeros that are inside the unit circle to $L_n^{(k)}$ is small. 

A better estimate than \eqref{eq2.6} appears in Section  3.3 page 14 of \cite{bat}, but with a more restricted range of $n$ in terms of $k$. It states that 
\begin{align}\label{lk_b1}
	\left| f_k(\alpha)(2\alpha - 1)\alpha^{n - 1}-3\cdot 2^{n-2}\right| < 3\cdot 2^{n-2}\cdot \frac{36}{2^{k/2}},\qquad {\text{\rm provided}}\qquad n<2^{k/2}.
\end{align}

\subsection{Linear forms in logarithms}
We use Baker--type lower bounds for nonzero linear forms in logarithms of algebraic numbers. There are many such bounds mentioned in the literature but we use one of Matveev from \cite{MAT}. Before we can formulate such inequalities, we need the notion of height of an algebraic number recalled below.  

\begin{definition}
	Let $ \gamma $ be an algebraic number of degree $ d $ with minimal primitive polynomial over the integers $$ a_{0}x^{d}+a_{1}x^{d-1}+\cdots+a_{d}=a_{0}\prod_{i=1}^{d}(x-\gamma^{(i)}), $$ where the leading coefficient $ a_{0} $ is positive. Then, the logarithmic height of $ \gamma$ is given by $$ h(\gamma):= \dfrac{1}{d}\Big(\log a_{0}+\sum_{i=1}^{d}\log \max\{|\gamma^{(i)}|,1\} \Big). $$
\end{definition}
In particular, if $ \gamma$ is a rational number represented as $\gamma=p/q$ with coprime integers $p$ and $ q\ge 1$, then $ h(\gamma ) = \log \max\{|p|, q\} $. 
The following properties of the logarithmic height function $ h(\cdot) $ will be used in the rest of the paper without further reference:
\begin{equation}\nonumber
	\begin{aligned}
		h(\gamma_{1}\pm\gamma_{2}) &\leq h(\gamma_{1})+h(\gamma_{2})+\log 2;\\
		h(\gamma_{1}\gamma_{2}^{\pm 1} ) &\leq h(\gamma_{1})+h(\gamma_{2});\\
		h(\gamma^{s}) &= |s|h(\gamma)  \quad {\text{\rm valid for}}\quad s\in \mathbb{Z}.
	\end{aligned}
\end{equation}
In Section 3, equation (12) of \cite{Brl} these properties were used to show the following inequality:
\begin{align}\label{eq2.9}
	h\left(f_k(\alpha)\right)<3\log k, ~~\text{for all}~~k\ge 2.
\end{align}

A linear form in logarithms is an expression
\begin{equation*}
	\Lambda:=b_1\log \gamma_1+\cdots+b_t\log \gamma_t,
\end{equation*}
where $\gamma_1,\ldots,\gamma_t$ are positive real  algebraic numbers and $b_1,\ldots,b_t$ are integers. We assume, $\Lambda\ne 0$. We need lower bounds 
for $|\Lambda|$. We write ${\mathbb K}:={\mathbb Q}(\gamma_1,\ldots,\gamma_t)$ and $D$ for the degree of ${\mathbb K}$ over ${\mathbb Q}$.
We start with the general form due to Matveev, see Theorem 9.4 in \cite{matl}. 

\begin{theorem}[Matveev, see Theorem 9.4 in \cite{matl}]
	\label{thm:Matl} 
	Put $\Gamma:=\gamma_1^{b_1}\cdots \gamma_t^{b_t}-1=e^{\Lambda}-1$. Assume $\Gamma\ne 0$. Then 
	$$
	\log |\Gamma|>-1.4\cdot 30^{t+3}\cdot t^{4.5} \cdot D^2 (1+\log D)(1+\log B)A_1\cdots A_t,
	$$
	where $B\ge \max\{|b_1|,\ldots,|b_t|\}$ and $A_i\ge \max\{Dh(\gamma_i),|\log \gamma_i|,0.16\}$ for $i=1,\ldots,t$.
\end{theorem}

In our application of Theorem \ref{thm:Matl}, we need to ensure that the linear forms in logarithms are indeed nonzero. To ensure this, we shall need the following result given as Lemma 2.8 in \cite{GGL1}.

\begin{lemma}[Lemma 2.8 in \cite{GGL1}]\label{lemGLm}
	Let \( N := N_{\mathbb{K}/\mathbb{Q}}, \) where \( \mathbb{K} = \mathbb{Q}(\alpha) \). Then
	\begin{enumerate}[(i)]
		\item For \( n, m \geq 1 \) and \( k \geq 2 \), \( |N(\alpha)| = 1 \).
		\item \( N(2\alpha - 1) = 2^{k+1} - 3 \) and \( N(f_k(\alpha)) = (k - 1)^2 / (2^{k+1}k^k - (k + 1)^{k+1}) \).
		\item For \( k \geq 3 \), \( N((2\alpha - 1)f_k(\alpha)) < 1 \).
	\end{enumerate}	
\end{lemma}

During the calculations, upper bounds on the variables are obtained which are too large, thus there is need to reduce them. To do so, we use some results from
approximation lattices and the so-called LLL--reduction method from \cite{LLL}. We explain this in the following subsection.

\subsection{Reduced Bases for Lattices and LLL--reduction methods}\label{sec2.3}

Let $k$ be a positive integer. A subset $\mathcal{L}$ of the $k$-dimensional real vector space ${ \mathbb{R}^k}$ is called a lattice if there exists a basis $\{b_1, b_2, \ldots, b_k \}$ of $\mathbb{R}^k$ such that
\begin{align*}
	\mathcal{L} = \sum_{i=1}^{k} \mathbb{Z} b_i = \left\{ \sum_{i=1}^{k} r_i b_i \mid r_i \in \mathbb{Z} \right\}.
\end{align*}
We say that $b_1, b_2, \ldots, b_k$ form a basis for $\mathcal{L}$, or that they span $\mathcal{L}$. We
call $k$ the rank of $ \mathcal{L}$. The determinant $\text{det}(\mathcal{L})$, of $\mathcal{L}$ is defined by
\begin{align*}
	\text{det}(\mathcal{L}) = | \det(b_1, b_2, \ldots, b_k) |,
\end{align*}
with the $b_i$'s being written as column vectors. This is a positive real number that does not depend on the choice of the basis (see \cite{Cas}, Section 1.2).

Given linearly independent vectors $b_1, b_2, \ldots, b_k $ in $ \mathbb{R}^k$, we refer back to the Gram--Schmidt orthogonalization technique. This method allows us to inductively define vectors $b^*_i$ (with $1 \leq i \leq k$) and real coefficients $\mu_{i,j}$ (for $1 \leq j \leq i \leq k$). Specifically,
\begin{align*}
	b^*_i &= b_i - \sum_{j=1}^{i-1} \mu_{i,j} b^*_j,~~~
	\mu_{i,j} = \dfrac{\langle b_i, b^*_j\rangle }{\langle b^*_j, b^*_j\rangle},
\end{align*}
where \( \langle \cdot , \cdot \rangle \)  denotes the ordinary inner product on \( \mathbb{R}^k \). Notice that \( b^*_i \) is the orthogonal projection of \( b_i \) on the orthogonal complement of the span of \( b_1, \ldots, b_{i-1} \), and that \( \mathbb{R}b_i \) is orthogonal to the span of \( b^*_1, \ldots, b^*_{i-1} \) for \( 1 \leq i \leq k \). It follows that \( b^*_1, b^*_2, \ldots, b^*_k \) is an orthogonal basis of \( \mathbb{R}^k \). 
\begin{definition}
	The basis $b_1, b_2, \ldots, b_n$ for the lattice $\mathcal{L}$ is called reduced if
	\begin{align*}
		\| \mu_{i,j} \| &\leq \frac{1}{2}, \quad \text{for} \quad 1 \leq j < i \leq n,~~
		\text{and}\\
		\|b^*_{i}+\mu_{i,i-1} b^*_{i-1}\|^2 &\geq \frac{3}{4}\|b^*_{i-1}\|^2, \quad \text{for} \quad 1 < i \leq n,
	\end{align*}
	where $ \| \cdot \| $ denotes the ordinary Euclidean length. The constant $ {3}/{4}$ above is arbitrarily chosen, and may be replaced by any fixed real number $ y $ in the interval ${1}/{4} < y < 1$ {\rm(see \cite{LLL}, Section 1)}.
\end{definition}
Let $\mathcal{L}\subseteq\mathbb{R}^k$ be a $k-$dimensional lattice  with reduced basis $b_1,\ldots,b_k$ and denote by $B$ the matrix with columns $b_1,\ldots,b_k$. 
We define
\[
l\left( \mathcal{L},y\right)= \left\{ \begin{array}{c}
	\min_{x\in \mathcal{L}}||x-y|| \quad  ;~~ y\not\in \mathcal{L}\\
	\min_{0\ne x\in \mathcal{L}}||x|| \quad  ;~~ y\in \mathcal{L}
\end{array}
\right.,
\]
where $||\cdot||$ denotes the Euclidean norm on $\mathbb{R}^k$. It is well known that, by applying the
LLL--algorithm, it is possible to give in polynomial time a lower bound for $l\left( \mathcal{L},y\right)$, namely a positive constant $\delta$ such that $l\left(\mathcal{L},y\right)\ge \delta$ holds (see \cite{SMA}, Section V.4).
\begin{lemma}[\cite{SMA}, Section V.4]\label{lem2.5m}
	Let $b_1, \dots, b_k$ be an LLL-reduced basis for a lattice $\mathcal{L}$ and $b_1^*, \dots, b_k^*$ be the corresponding Gram-Schmidt orthogonal basis. Let $y\in\mathbb{R}^k$ and $z=B^{-1}y$.
	\begin{enumerate}[(i)]
		\item If $y\not \in \mathcal{L}$, let $i_0$ be the largest index such that $z_{i_0}\ne 0$ and put $\lambda:=\{z_{i_0}\}$.
		\item If $y\in \mathcal{L}$, put $\lambda:=1$.
	\end{enumerate}
	Now, define
	\[
	c_1 := \max_{1\le j\le k}\left\{\dfrac{\|b_1\|}{\|b_j^*\|}\right\}.
	\]
	Then, $l(\mathcal{L}, y) \ge  \delta = \lambda\|b_1\|c_1^{-1}$.
\end{lemma}

In our application, we are given real numbers $\eta_0,\eta_1,\ldots,\eta_k$ which are linearly independent over $\mathbb{Q}$ and two positive constants $c_3$ and $c_4$ such that 
\begin{align}\label{2.9m}
	|\eta_0+x_1\eta_1+\cdots +x_k \eta_k|\le c_3 \exp(-c_4 H),
\end{align}
where the integers $x_i$ are bounded as $|x_i|\le X_i$ with $X_i$ given upper bounds for $1\le i\le k$. We write $X_0:=\max\limits_{1\le i\le k}\{X_i\}$. The basic idea in such a situation, from \cite{Weg}, is to approximate the linear form \eqref{2.9m} by an approximation lattice. So, we consider the lattice $\mathcal{L}$ generated by the columns of the matrix
$$ \mathcal{A}=\begin{pmatrix}
	1 & 0 &\ldots& 0 & 0 \\
	0 & 1 &\ldots& 0 & 0 \\
	\vdots & \vdots &\vdots& \vdots & \vdots \\
	0 & 0 &\ldots& 1 & 0 \\
	\lfloor C\eta_1\rfloor & \lfloor C\eta_2\rfloor&\ldots & \lfloor C\eta_{k-1}\rfloor& \lfloor C\eta_{k} \rfloor
\end{pmatrix} ,$$
where $C$ is a large constant usually of the size of about $X_0^k$ . Let us assume that we have an LLL--reduced basis $b_1,\ldots, b_k$ of $\mathcal{L}$ and that we have a lower bound $l\left(\mathcal{L},y\right)\ge \delta$ with $y:=(0,0,\ldots,-\lfloor C\eta_0\rfloor)$. Note that $ \delta$ can be computed by using the results of Lemma \ref{lem2.5m}. Then, with these notations the following result  is Lemma VI.1 in \cite{SMA}.
\begin{lemma}[Lemma VI.1 in \cite{SMA}]\label{lem2.6m}
	Let $S:=\displaystyle\sum_{i=1}^{k-1}X_i^2$ and $T:=\dfrac{1+\sum_{i=1}^{k}X_i}{2}$. If $\delta^2\ge T^2+S$, then inequality \eqref{2.9m} implies that we either have $x_1=x_2=\cdots=x_{k-1}=0$ and $x_k=-\dfrac{\lfloor C\eta_0 \rfloor}{\lfloor C\eta_k \rfloor}$, or
	\[
	H\le \dfrac{1}{c_4}\left(\log(Cc_3)-\log\left(\sqrt{\delta^2-S}-T\right)\right).
	\]
\end{lemma}

Finally, we present an analytic argument which is Lemma 7 in \cite{GL}.  
\begin{lemma}[Lemma 7 in \cite{GL}]\label{Guz} If $ r \geq 1 $, $T > (4r^2)^r$ and $T >  \dfrac{p}{(\log p)^r}$, then $$p < 2^r T (\log T)^r.$$	
\end{lemma}
SageMath 10.6 is used to perform all computations in this work.

\section{Proof of Theorem \ref{thm1.1}.}\label{Sec3}
In this section, we prove Theorem \ref{thm1.1}. We do this in two separate cases depending on $n$ versus $k$.
\subsection{The case $2\le n\le k$}
In this case, we have $L_n^{(k)} = 3 \cdot 2^{n-2}$. Comparing \eqref{eq2.2} with \eqref{eq:main}, we have $3 \cdot 2^{n-2}=(2^{\mathfrak{p}} + 1)/3$, or simply
\begin{align}\label{eq:nk}
	 9 \cdot 2^{n-2}=2^{\mathfrak{p}} + 1.
\end{align}
If $n=2$, then $9 \cdot 2^{2-2}=2^{\mathfrak{p}} + 1$, from which $\mathfrak{p}=3$ for all $k\ge 2$. This solution is stated in the main result. If $n\ge 3$, then the left-hand side of \eqref{eq:nk} is even while the right-hand side is odd. This means that \eqref{eq:nk} has no solutions in the interval $3\le n\le k$. From now on, we assume $n>k$.

\subsection{The case $n> k$}
In this case, we assume $n>k$. Since $k\ge 2$, then we can also assume that $n\ge 3$ for the remaining part of the proof.
\subsubsection{Bounding $n$  in terms of $k$}
We begin by proving the following result.
\begin{lemma}\label{lem3.1}
	Let $n$, $k$, $\mathfrak{p}$ be integer solutions to Eq. \eqref{eq:main} with $n>k\ge 2$ and $\mathfrak{p}\ge 3$ be a prime number, then	
\begin{align*}
n<2.0\cdot 10^{15}  k^4(\log k)^3.
\end{align*}	
\end{lemma}
\begin{proof}
	We go back to \eqref{eq:main} and rewrite it using the Binet formula in \eqref{eq:Lnwitherror} as
	\begin{align*}
		L_n^{(k)}-f_k(\alpha)(2\alpha-1)\alpha^{n-1}&=e_k(n),\\
		\dfrac{2^{\mathfrak{p}} + 1}{3}-f_k(\alpha)(2\alpha-1)\alpha^{n-1}&=e_k(n),\\
		\dfrac{2^{\mathfrak{p}}}{3}-f_k(\alpha)(2\alpha-1)\alpha^{n-1}&=e_k(n)-\dfrac{1}{3}.
	\end{align*}	
Taking absolute values, we get
	\begin{align*}
	\left| \dfrac{2^{\mathfrak{p}}}{3}-f_k(\alpha)(2\alpha-1)\alpha^{n-1}\right|&<\dfrac{3}{2}+\dfrac{1}{3}=\dfrac{11}{6},
\end{align*}
and dividing both sides by $f_k(\alpha)(2\alpha-1)\alpha^{n-1}>0$, gives
	\begin{align*}
	\left| 2^{\mathfrak{p}} (3(2\alpha-1)f_k(\alpha))^{-1}\alpha^{-(n-1)}-1\right|&<\dfrac{11}{6f_k(\alpha)(2\alpha-1)\alpha^{n-1}}\\
	&< \dfrac{11}{3}\alpha^{-n}.
\end{align*}
In the above simplification, we have used the fact that $f_k(\alpha)>1/2$ from \eqref{eq2.5} and $\alpha>2(1-2^{-k})\ge 1.5$ for all $k\ge 2$ from \eqref{eq2.3}. Therefore,
	\begin{align}\label{g3}
		|\Gamma_1|:=\left| 2^{\mathfrak{p}}(3(2\alpha-1)f_k(\alpha))^{-1}\alpha^{-(n-1)}-1\right|
		&< \dfrac{11}{3}\alpha^{-n}.
	\end{align}
	Notice that $\Gamma_1\ne 0$, otherwise we would have 
	\begin{align*}
		(2\alpha-1)f_k(\alpha)\alpha^{n-1}= \dfrac{2^{\mathfrak{p}} }{3}.
	\end{align*}	
Taking norms in ${\mathbb K}={\mathbb Q}(\alpha)$ and using $|N(\alpha)|=1$, and item $(ii)$ of Lemma \ref{lemGLm}, the above equation becomes
\begin{align*}
	N\left(f_k(\alpha)\right)\cdot N(2\alpha-1)=\left(\dfrac{2^{\mathfrak{p}}}{3}\right)^{k}.
\end{align*}
This implies that
\begin{align*}
	1\ge \dfrac{(k - 1)^2 }{ 2^{k+1}k^k - (k + 1)^{k+1}}\cdot \left(2^{k+1}-3\right)=\left(\dfrac{2^{\mathfrak{p}}}{3}\right)^{k}\ge \left(\dfrac{2^{3}}{3}\right)^{2}=\dfrac{64}{9},
\end{align*}
which is absurd. 
	
The algebraic number field containing the following $\gamma_i$'s is $\mathbb{K} := \mathbb{Q}(\alpha)$. We have $D = k$, $t :=3$,
\begin{alignat*}{3}
	\gamma_{1}&:=2 ,    \qquad    &\gamma_{2}&:=3(2\alpha-1)f_k(\alpha),         \qquad& \gamma_{3} &:= \alpha,      \\
	b_1         &:=\mathfrak{p},        &\quad b_2    &:= -1,         &\quad b_3    &:=-(n-1).
\end{alignat*}
Notice that $h(\gamma_{1})=h(2)= \log 2 $, $h(\gamma_{3})=(\log \alpha)/k <0.7/k$, so we take $A_{1}:=k\log 2$ and $A_{3}:=0.7$. 
For $A_2$, we first compute 
\begin{align*}
	h(\gamma_{2})&:=h\left(3(2\alpha-1)f_k(\alpha)\right)\\
	&\le h(3)+ 	h\left(2\alpha-1\right)+h\left(f_k(\alpha)\right)\\
	&\le h(3)+ 	h(2) +h\left(\alpha\right)+2\log 2+h\left(f_k(\alpha)\right)\\
	&<\log 3+3\log 2+\frac{\log\alpha}{k}+3\log k\\
	&<9\log k,
\end{align*}
for all $k\ge 2$. In the above computations, we have used inequality \eqref{eq2.9}. So, we can take $A_2:=9k\log k$. Next, $B \geq \max\{|b_i|:i=1,2,3\}$, and by relation \eqref{m_b}, we can take $B:=n+3$.

Now, Theorem \ref{thm:Matl} gives,
	\begin{align}\label{eq3.12}
		\log |\Gamma_1| &> -1.4\cdot 30^{6} \cdot 3^{4.5}\cdot k^2 (1+\log k)(1+\log (n+3))\cdot k\log 2\cdot   0.7\cdot 9k\log k\nonumber\\
		&> -3.0\cdot 10^{12}  k^4(\log k)^2 \log (n+3).
	\end{align}
	Comparing \eqref{g3} and \eqref{eq3.12}, we get
	\begin{align*}
		n\log \alpha-\log(11/3) &< 3.0\cdot 10^{12}  k^4(\log k)^2 \log (n+3)\\
		&< 6.0\cdot 10^{12}  k^4(\log k)^2 \log n.
	\end{align*}
	Therefore, we obtain $n<2.0\cdot 10^{13}  k^4(\log k)^2 \log n$. 
	
	We now apply Lemma \ref{Guz} with  $p:=n$, $r:=1$, $T:=2.0\cdot 10^{13}  k^4(\log k)^2$ and have 
\begin{align*}
	n&<2\cdot 2.0\cdot 10^{13}  k^4(\log k)^2\log\left( 2.0\cdot 10^{13}  k^4(\log k)^2\right)\\
	&=4.0\cdot 10^{13} k^4(\log k)^2 \left(\log (2.0\cdot 10^{13})+4\log k+2\log\log k\right)\\
	&<4.0\cdot 10^{13} k^4(\log k)^3 \left(\dfrac{31}{\log k}+4+\dfrac{2\log\log k}{\log k}\right)\\
	&<2.0\cdot 10^{15} k^4(\log k)^3,
\end{align*}
for all $k\ge 2$. Thus, $n<2.0\cdot 10^{15} k^4(\log k)^3$ and the proof of Lemma \ref{lem3.1} is complete.
\end{proof}

\subsubsection{The case $k>190$}
Let us assume for a moment that $k > 190$.
By Lemma \ref{lem3.1}, we have 
\begin{align*}
	n<2.0\cdot 10^{15} k^4(\log k)^3 < 2^{k/2},
\end{align*}
where the last inequality holds true for all $k>190$. Since $n<2^{k/2}$, this puts us in position to use inequality \eqref{lk_b1}, which together with \eqref{eq2.6} gives
\begin{align*}
	\left|\dfrac{2^{\mathfrak{p}} + 1}{3}-3\cdot2^{n-2}\right| &=\left|L_n^{(k)} - 3\cdot2^{n-2}\right|\\
	&\le \left|L_n^{(k)} - f_k(\alpha)(2\alpha - 1)\alpha^{n - 1}\right|+\left|f_k(\alpha)(2\alpha - 1)\alpha^{n - 1} - 3\cdot2^{n-2}\right|\\
	&< \dfrac{3}{2} + 3\cdot 2^{n-2}\cdot \frac{36}{2^{k/2}}.
\end{align*}
Therefore,
\begin{align*}
	\left|\dfrac{2^{\mathfrak{p}}}{3}-3\cdot2^{n-2}\right| 
	&< \dfrac{3}{2} + 3\cdot 2^{n-2}\cdot \frac{36}{2^{k/2}}+\dfrac{1}{3}\\
	&= \dfrac{11}{6} + 3\cdot 2^{n-2}\cdot \frac{36}{2^{k/2}}\\
	&< 2 +  2^{n}\cdot \frac{27}{2^{k/2}}\\
&< 2^{n+1}\cdot \frac{27}{2^{k/2}},
\end{align*}
since we are in the case $n>k  >k/2$. Multiplying through by 3 and dividing through by $2^{n-2}$, we get
\begin{align}\label{2p}
	\left|9-2^{\mathfrak{p}-n+2}\right|<  \frac{648}{2^{k/2}}.
\end{align}
Notice that since $n>k$, we can rewrite \eqref{eq:main} using \eqref{eq:32} as 
\begin{align*}
	 \frac{2^{\mathfrak{p}}}{3}<\frac{2^{\mathfrak{p}} + 1}{3}=L_n^{(k)} <3\cdot 2^{n-2}.
\end{align*}
This gives $2^{\mathfrak{p}}<9\cdot 2^{n-2}$ and taking logarithms both sides simplifies it to $\mathfrak{p}<n+2$. Therefore, $\mathfrak{p}-n+2<4$. Note also that $\mathfrak{p}-n+2$ is an integer, so $\mathfrak{p}-n+2 \le 3$ and hence $1\le \left|9-2^{\mathfrak{p}-n+2}\right| < 9 $. Comparing this inequality with \eqref{2p}, we get
\begin{align*}
	1\le \left|9-2^{\mathfrak{p}-n+2}\right|<  \frac{648}{2^{k/2}},
\end{align*}
from which we deduce that $2^{k/2}<648$. This gives $k<19$, contradicting the assumption that $k>190$. Hence, Equation \eqref{eq:main} has no solutions whenever $k>190$.

To proceed, we first recall one additional simple fact from calculus. If $x\in \mathbb{R}$ satisfies $|x|<1/2$, then 
\begin{align}\label{eq2.5g}
	|\log(1+x)|&<|x-x^2/2+-\dots|
	<|x|+\frac{|x|^2+|x|^3+\dots}2
	<|x|\left(1+\frac{|x|}{2(1-|x|)}\right)<\frac 32 |x|.
\end{align}
 We use this inequality in the next subsection.

\subsubsection{The case $k\le 190$}
If $k \leq 190$, it follows from \eqref{lem3.1} that  
\[
n < 2.0 \cdot 10^{15} k^4 (\log k)^3 < 1.0 \cdot 10^{29}.
\]
Our goal here is to derive a smaller upper bound on $n$ for further effective computation. To accomplish this, we revisit \eqref{g3} and write
\[
\Gamma_1 := 2^{\mathfrak{p}}(3(2\alpha-1)f_k(\alpha))^{-1}\alpha^{-(n-1)}-1 = e^{\Lambda_1} - 1.
\]
Since we established that $\Gamma_1 \neq 0$, it follows that $\Lambda_1 \neq 0$. Now, assuming $\ell \geq 2$, we obtain the inequality  
\[
\left| e^{\Lambda_1} - 1 \right| = |\Gamma_1| < 0.5,
\]
which leads to $ |\log (1+\Gamma_1)| < 1.5|\Gamma_1|$, via \eqref{eq2.5g}. Consequently, we derive  
\[
\left| (n - 1) \log \alpha - \mathfrak{p} \log 2 + \log \left( 3(2\alpha-1)f_k(\alpha) \right) \right| < \frac{11}{2}\alpha^{-n}.
\]
Now, we apply the LLL-algorithm for each $k \in [2,190]$ to obtain a lower bound for the smallest nonzero value of the above linear form, constrained by integer coefficients with absolute values not exceeding $n+3 < 1.1 \cdot 10^{29}$. Specifically, we consider the lattice  
\[
\mathcal{A} = \begin{pmatrix} 
	1 & 0 & 0 \\ 
	0 & 1 & 0 \\ 
	\lfloor C\log \alpha\rfloor & \lfloor C\log (1/2)\rfloor & \lfloor C\log \left(3(2\alpha-1)f_k(\alpha)\right) \rfloor
\end{pmatrix},
\]
where we set $C := 4.0\cdot 10^{87}$ and $y := (0,0,0)$. Applying Lemma \ref{lem2.5m}, we obtain 
\[
c_1 = 10^{-67} \quad \text{and} \quad \delta = 2.83\cdot 10^{63}. 
\]
Using Lemma \ref{lem2.6m}, we obtain $S =2.42 \cdot 10^{58}$ and $T = 1.65 \cdot 10^{29}$. Since $\delta^2 \geq T^2 + S$, then choosing $c_3 := 11/2$ and $c_4 := \log \alpha$, we establish the bound $n \leq 141$. 

To conclude this case, we look for solutions to \eqref{eq:main} in the reduced ranges $k\in [2, 190]$, $n \in [k + 1, 141]$ and $\mathfrak{p} \in [3, n+3]$. A computational search in SageMath 10.6 reveals only two integral solutions to \eqref{eq:main} which are stated in Theorem \ref{thm1.1}.  \qed

\section*{Acknowledgments} 
The author thanks the Mathematics division of Stellenbosch University for funding his PhD studies.

\section*{Addresses}

$ ^{1} $ Mathematics Division, Stellenbosch University, Stellenbosch, South Africa.

Email: \url{hbatte91@gmail.com}
\end{document}